\documentclass[a4paper,12pt]{amsart}

\usepackage{latexsym}
\usepackage{amssymb,amsmath}
\usepackage{graphics}
\usepackage{url}
\usepackage[vcentermath]{youngtab}
  
\usepackage{booktabs}  
\usepackage{xcolor}
\usepackage{fancyhdr}
\usepackage{tikz}
\usetikzlibrary{matrix}

\usepackage{caption}
\usepackage{tikz}
\usetikzlibrary{decorations}
\usetikzlibrary{arrows}
\tikzstyle{empty}=[circle,draw=black!80,thick]
\tikzstyle{emptyn}=[circle,draw=black!80,fill=white,scale=0.5] 
\tikzstyle{nero}=[circle,draw=black!80,fill=black!80,thick]

\newtheorem{teo}{Theorem}[section]
\newtheorem{corollary}[teo]{Corollary}

\newtheorem{proposition}[teo]{Proposition}

\newtheorem{lemma}[teo]{Lemma}

\theoremstyle{definition}
\newtheorem{definition}[teo]{Definition}

\newtheorem{example}[teo]{Example}

\newtheorem {Step}    {Step}
\newtheorem *{FStep}    {Final Step}

\newcommand{\fS}{{\mathfrak{S}}}

\newcommand{\nor}{\unlhd}
\DeclareMathOperator{\GL}{GL}

\DeclareMathOperator{\Aut}{Aut}
\DeclareMathOperator{\Gal}{Gal}

\def\syl#1#2{{\rm Syl}_#1(#2)}
\def\cent#1#2{{\bf C}_{#1}(#2)}

\def\ch#1{{\rm Char}(#1)}
\def\irr#1{{\rm Irr}(#1)}

\def\irrp#1#2{{\rm Irr}_{#2'}(#1)}

\def\norm#1#2{{\bf N}_{#1}(#2)}

\newcommand{\op}{\operatorname}

\newcommand{\Q}{\mathbb Q}

\title[]{Alperin-McKay natural correspondences in solvable and symmetric groups for the prime $p=2$}
\author{Eugenio Giannelli}
\address[E. Giannelli]{Trinity Hall, University of Cambridge, Trinity Lane, CB21TJ, UK}
\email{eg513@cam.ac.uk}
\author{John Murray}
\address[John Murray]{Department of Mathematics \& Statistics, National University of Ireland Maynooth,
Co. Kildare, Ireland}
\email{john.murray@maths.nuim.ie}
\author{Joan Tent}
\address[J. Tent]{Departament de Matem\`atiques, Universitat de Val\`encia, 46100 Burjassot, Val\`encia, Spain}
\email{joan.tent@uv.es}
\date{February 20, 2017}
\begin{document}
\thanks{The first author's research was funded by Trinity Hall, University of Cambridge and by Gunter Malle's ERC Advanced Grant 291512. The third author 
has been supported by 
MTM2014-53810-C2-01 of the Spanish MEyC and by the Prometeo/Generalitat Valenciana.}

\begin{abstract}
Let $G$ be a finite solvable or symmetric group and let $B$ be a $2$-block of $G$. We construct a canonical correspondence between the irreducible characters of height zero in $B$ and those in its Brauer first main correspondent.
For symmetric groups our bijection is compatible with restriction of characters. 
\end{abstract}

\maketitle
\thispagestyle{empty}

\section{Introduction}

Let $\irr G$ be the set of irreducible characters of a finite group $G$, and let $\op{Irr}_{p'}(G)$ be the subset of characters of $p'$-degree, for any prime $p$. The {\em McKay conjecture} (recently proved for $p=2$ in \cite{MS}) asserts that $|\op{Irr}_{p'}(G)|=|\op{Irr}_{p'}(\norm G P)|$ when $P$ is a Sylow $p$-subgroup of $G$. 
In \cite{brown} Isaacs defined a natural correspondence of characters for the McKay conjecture in solvable groups at the prime $p=2$. As one expects for natural correspondences, the bijections of characters in \cite{brown} 
preserve fields of values of characters. Actually, the hypothesis in \cite{brown} are more general, and apply to odd-order groups at any prime. However, natural McKay bijections do not always exist, even for solvable groups when $p\ne2$ (e.g. $G=\GL(2,3)$ has two non-real irreducible characters of degree 2, while all irreducible characters of 
the normalizer of a Sylow $3$-subgroup of $G$ are rational-valued, so no bijections preserving fields of values exist in this case).

Next recall that $\irr G$ is partitioned into $p$-blocks; we let $\irr B$ be the characters belonging to a given $p$-block $B$. The defect of $B$ is the maximum non-negative integer $d$ such that $|G|/p^d\chi(1)$ is an integer, for some $\chi\in\irr B$. So if $\chi\in\irr {B}$, the $p$-part of $|G|/p^d\chi(1)$ is $p^h$, where $h\geq0$ is the height of $\chi$. We use $\mathrm{Irr}_{0}(B)$ to denote the characters of height 0 in $\mathrm{Irr}(B)$. A defect group of $B$ is a certain $p$-subgroup $D$ of $G$ such that $|D|=p^d$. The defect groups of $B$ are determined up to $G$-conjugacy.

The \textit{Alperin-McKay conjecture}, first stated in \cite{Alp}, asserts that $|\mathrm{Irr}_{0}(B)|=|\mathrm{Irr}_{0}(b)|$, where $b\in {\rm Bl}(\norm G D)$ is the Brauer first main correspondent of $B$. T.~Okuyama and M.~Wajima proved the Alperin-McKay conjecture for $p$-solvable groups in \cite{OW}. However, as with the McKay conjecture, it is generally not possible to define natural correspondences in the context of the Alperin-McKay conjecture, even for solvable groups. 


In this article we take $p=2$. Our first main theorem extends Isaacs results \cite{brown}
to
establish a natural Alperin-McKay bijection for solvable groups.

\begin{teo}\label{AMK}
Let $G$ be a finite solvable group, let $B$ be a $2$-block of\/ $G$ which has defect group $D$, and let $b$ be the $2$-block of\/ $\norm G D$ which is the Brauer correspondent of\/ $B$. Then there exists a natural correpondence between the height-zero irreducible characters in $B$ and  the height-zero irreducible characters in $b$. 
\end{teo}

As a consequence of this theorem, if $b$ is the Brauer correspondent of a $2$-block $B$ of
a solvable group $G$, then the fields of values of the height-zero irreducible characters in
$B$ are the same as the fields of values of the height-zero irreducible characters in $b$.



\medskip

In general, it is particularly rare to find a natural correspondence that is compatible with restriction of characters, in the context of the Alperin-McKay conjecture.
Surprisingly this happens when $G=\fS_n$, the symmetric group on $n$ letters. 
The second main result of the article can be stated as follows. 

\begin{teo}\label{Thm:main}
Let $B$ be a $2$-block of\/ $\fS_n$ with defect group $D$, and let $b$ be the $2$-block of\/ $\norm {\fS_n} D$ which is the Brauer correspondent of\/ $B$. Then there exists a natural correpondence between the height-zero irreducible characters in $B$ and  the height-zero irreducible characters in $b$. If $\alpha\in\op{Irr}_0(B)$ corresponds to $\beta\in\op{Irr}_0(b)$, then $[\alpha{\downarrow_{\norm {\fS_{n}} D}},\beta]\ne0$.
\end{teo}

Theorem \ref{Thm:main} is a generalization of \cite[Theorem 4.3]{GKNT} where, building on \cite{G}, it was possible to explicitly construct a canonical correspondence of characters, compatible with restriction, in the context of the McKay conjecture for $\fS_n$. 

The last main result of the paper is Theorem \ref{theo:restriction}. There we study in more details the restriction to $\norm {\fS_{n}} D$ of any $\chi\in\mathrm{Irr}(B)$. In particular we show that for every $\gamma\in\mathrm{Irr}(B)$ there exists $\beta\in\mathrm{Irr}_0(b)$ such that $\beta$ is a constituent of $\gamma{{\downarrow_{\norm {\fS_{n}} D}}}$.


\section{Solvable groups}\label{secsolv}

We first recall the Glauberman correspondence. Suppose that $P$ and $G$ are finite groups such that $P$ acts on $G$ and $(|P|,|G|)=1$. In particular we may form the semi-direct product $G\rtimes P$. Glauberman's Lemma 13.8 of \cite{isaacs} asserts that if $X$ is a $(G\rtimes P)$-set, such that $G$ acts transitively on $X$, then $P$ fixes an element of $X$. One useful consequence of this is that if $N$ is a $P$-invariant normal subgroup of $G$ then $\cent {G/N} P=\cent G PN/N$.

Now suppose that $P$ is solvable. Let ${\rm Irr}_P(G)$ denote the set of $P$-invariant irreducible characters of $G$.
The Glauberman correspondence is a uniquely defined bijection
$$
\, {}^*\, :{\rm Irr}_P(G)\longrightarrow \irr {\cent G P}.
$$
Our notation $\, ^*\, $ supresses the dependence of this map on the action of $P$ on $G$, as it should be clear from the context. In the key case when $P$ is a $p$-group, we need the following property of Glauberman's map:

\begin{lemma}\label{glauberman}
Let $P$ be a finite $p$-group which acts coprimely on a finite group $G$ and
let $V$ be a $P$-invariant subgroup of $G$ which contains $\cent G P$. If $\chi\in {\rm Irr}_P(G)$, 
there is a unique $\psi\in{\rm Irr}_P(V)$ such that $p\nmid[\chi{\downarrow_V}, \psi]$.
Furthermore, $\chi^*=\psi^*$.
\end{lemma}

\begin{proof}
This is Corollary $3.3$ of \cite{wolf}.
\end{proof}

The reader is referred to Chapter 13 of \cite{isaacs} for a more thorough discussion of coprime actions and the Glauberman correspondence.

Now we discuss our plan of attack for the proof of Theorem \ref{AMK}, which consists first in reducing the general situation to the maximal defect case. Let $G$ be a finite solvable group and let $B$ be a $2$-block of $G$. We use induction on the order of $G$, so assume Theorem \ref{AMK} holds for any solvable group of order strictly smaller than $|G|$. 
Write $M=O_{2'}(G)$, and note that (by the covering theory of blocks) there exists $\theta\in\irr M$, unique up to $G$-conjugacy, such that $\irr B\subseteq \irr {G\, |\, \theta}$. Let $T=I_G(\theta)$ be the stabilizer of $\theta$ in $G$. By Fong-Reynolds Theorem 9.14 of \cite{nav}, there exists a unique block $B_0$ of $T$ whose irreducible characters are the Clifford correspondents (with respect to $\theta$) of the irreducible characters in $B$. Furthermore, this correspondence is height-preserving. So induction from $T$ to $G$ establishes a  bijection from $\op{Irr}_0(B_0)$ to $\op{Irr}_0(B)$, and a defect group of $B_0$ is a defect group of $B$. In particular we can assume that $D\subseteq T$.

Let $\theta^*\in\irr {\cent M D}$ be the Glauberman correspondent of $\theta$ with respect to the action of $D$ on $M$. Note that $\cent M D=N\cap M$, where $N=\norm G D$. By uniqueness in Glauberman's correspondence, $N\cap T$ is the stabilizer of $\theta^*$ in
$N$. Let $b_0\in\op{Bl}(N\cap T)$ be the Brauer first main correspondent of $B_0$. Then $b$ lies over $\theta^*$, and
the irreducible characters lying in $b_0$ are precisely the Clifford correspondents (over $\theta^*$) of the irreducible characters lying in $b$ (see \cite{slattery} for details). As before, induction defines a height-preserving bijection from $\irr {b_0}$ into $\irr b$.

Suppose that $T<G$. Our inductive hypothesis yields a natural correspondence from the set of height-zero characters in $B_0$ into the set of height-zero characters in $b_0$. In this case, induction of characters respectively from $T$ into $G$, and from $N\cap T$ into $N$, establishes the desired bijection between the height-zero characters in $B$ and the height-zero characters in $b$. So we may assume that $T=G$. 

By the previous paragraphs, we may suppose that $\theta$ is $G$-invariant. So $D\in\syl 2 G$ and $\irr B=\irr{G\, |\, \theta}$, by Fong's Theorem 10.20 of \cite{nav}. Since $\irr{b}=\irr{N\, |\, \theta^*}$ in this case, to complete the proof we need to prove a certain compatibility condition between the correspondence defined by Isaacs in \cite{brown} for the McKay conjecture
in solvable groups and the prime $p=2$, and Glauberman correspondence (see Theorem \ref{main} below). Working towards this, we begin with a bijection of characters defined by Isaacs in Theorem 10.6(ii) of \cite{brown}. What we really need to show is that this bijection has the property stated.
We note that this property may also be established using the work of Wolf \cite{wolf}, and that a related question was treated in \cite{lewis}. 

\begin{lemma}\label{bottom}
Let $G$ be a finite group, let $P\in\syl 2 G$, let $K,L\unlhd G$ with $L\subseteq K$ and let $\xi\in{\rm Irr}_P(L)$. Assume in addition that: 
\begin{enumerate}
\item $|K|$ is odd;
\item $K/L$ is abelian;
\item $G=\norm G P K$ and $\cent K P\subseteq L$;
\end{enumerate}
Set $H=\norm G P L$. Then there is a choice-free one-to-one correspondence:
$$
f_\xi\, :\irrp{G\, |\, \xi}{2}\longrightarrow \irrp{H|\, \xi}{2}\, . 
$$
This character correspondence satisfies the following property: let $V$ be a $P$-invariant odd-order subgroup of\/ $G$ which contains $K$ and
let $\varphi\in{\rm Irr}_P(V\mid\xi)$. Suppose that $\nu\in\irr {V\cap H\, |\, \xi}$ is the unique $P$-fixed character with $\varphi^*=\nu^*$. Then $f_\xi $ restricts to a one-to-one correspondence:
$$
f_\xi\, :\irrp{\, G|\, \varphi}{2}\longrightarrow \irrp{H |\, \nu}{2}\, .
$$
\end{lemma}

\begin{proof}
Our assumptions give $G=HK$, $L=H\cap K$ and thus $G/K\cong H/L$. We use induction on $|G|$. We shall denote the map $f_\xi$ in the statement of Lemma \ref{bottom} by
$$
\, \hat\,  :\irrp{G\, |\, \xi}{2}\longrightarrow \irrp{H |\, \xi}{2}\, . 
$$

\medskip

\begin{Step}\label{invariant}
{We can assume that
$\xi$ is invariant in $G$.}
\end{Step}

\begin{proof}[Proof of Step 1]
Let $T=I_G(\xi)$ be the inertia subgroup of $\xi$ in $G$, and assume that $T<G$. Now $P\subseteq T$, by our hypothesis. Thus, 
the result holds for the group $T$, its subgroups $K\cap T$, $L$ and $H\cap T$, and the character $\xi\in\irr L$ by induction. So there exists a choice-free correspondence

\begin{equation}\label{inertia}
\irrp{T\, |\, \xi}{2}\longrightarrow \irrp{H\cap T\, |\, \xi}{2}\, 
\end{equation} satisfying the desired property. 
Let $\alpha\in \irrp{T\, |\, \xi}{2}$, and note that $\chi=\alpha{\uparrow^G}\in\irrp{G\, |\, \xi}{2}$ by Clifford's correspondence. 
Suppose that $\beta\in\irrp{H\cap T\, |\, \xi}{2}$ corresponds to $\alpha$ under the map (\ref{inertia}). Then ${\beta}{\uparrow^{H}}\in\irrp{H\, |\, \xi}{2}$ by Clifford's correspondence, and we set $\hat{\chi}:={\beta}{\uparrow^{H}}$. By the properties of Clifford's correspondence, this defines a one-to-one map
$$
\hat{}\,:\,  \irrp{G\, |\, \xi}{2}\longrightarrow \irrp{H\, |\, \xi}{2}\, .
$$
Next we prove that the map $\, \hat{}\, $ satisfies the property in the statement. 

Let $K\subseteq V\subseteq G$, $\varphi\in\irr {V\, |\, \xi}$ and $\nu\in\irr{V\cap H\, |\, \xi}$ be as in the statement. 
Let also $\gamma\in\irr{V\cap T\, |\, \xi}$ be such that $\gamma{\uparrow^V}=\varphi$, and let $\delta\in\irr{V\cap H\cap T\, |\, \xi}$ such that
$\delta{\uparrow^{V\cap H}}=\nu$. Since both $\varphi$ and $\nu$ are $P$-invariant, it follows from the uniqueness in Clifford's
correspondence that $\gamma$ and $\delta$ are $P$-fixed as well.

We claim that $\gamma^*=\delta^*$. First, note that $\cent{V\cap T}{P}\subseteq V\cap H\cap T$.
Write $\mu=\gamma{\uparrow^{(V\cap T)K}}\in\op{Irr}{((V\cap T)K\, |\, \xi)}$, 
and notice that $\mu$ is $P$-invariant. 
By Lemma \ref{glauberman}, we have that $\mu^*=\gamma^*$. 
Now, let $\eta$ be the unique $P$-invariant irreducible constituent of 
$\gamma{\downarrow_{V\cap H\cap T}}$ having odd multiplicity, again by Lemma \ref{glauberman}, so 
$\eta^*=\gamma^*=\mu^*$. We shall prove that $\eta=\delta$. Note that since $\gamma$ lies over 
$\xi$, which is invariant in $V\cap T$, it is clear that $\eta$ lies over 
$\xi$ as well. Thus $\omega=\eta{\uparrow^{V\cap H}}$ is irreducible and $P$-invariant, 
by Clifford's correspondence. Also, 
observe that $\eta$ is the unique $P$-invariant
irreducible constituent of $\mu{\downarrow_{V\cap H\cap T}}$ having odd multiplicity, by Lemma \ref{glauberman}. 
Let $\zeta=\zeta_1,\ldots,\zeta_r$ be a non-trivial $P$-orbit
of irreducible constituents of $\mu{\downarrow_{V\cap H\cap T}}$, and write 
$\Delta=\zeta_1+\cdots+\zeta_r$. Suppose that 
$\zeta_i=\zeta^{x_i}$, where $x_i\in P$. 
Then, since $\omega$ is $P$-invariant, we have that

$$
[\omega, \Delta{\uparrow^{V\cap H}}]=\sum_{i=1}^r\, [\omega, \zeta_i\uparrow^{V\cap H}]=
\sum_{i=1}^r\, [\omega, (\zeta{\uparrow^{V\cap H}})^{x_i}]=
r\, [\omega, \zeta{\uparrow^{V\cap H}}]\, ,
$$ which of course is even.
It follows that

$$
[\omega, (\mu{\downarrow_{V\cap H\cap T}})\big{\uparrow}^{V\cap H}]\equiv [\eta, \mu{\downarrow_{V\cap H\cap T}}]\equiv 1 \pmod 2\, .
$$ Since $\varphi{\downarrow_{V\cap H}}=(\mu{\downarrow_{V\cap H\cap T}})\big{\uparrow}^{V\cap H}$, we deduce by Lemma \ref{glauberman} that
$\omega=\nu$. Thus $\eta=\delta$ by uniqueness in 
Clifford's correspondence, and our claim follows.

Now, by induction the map (\ref{inertia}) restricts to a one-to-one correspondence
$$
\irrp{T\, |\, \gamma}{2}\longrightarrow \irrp{H\cap T\, |\, \delta}{2}\, .
$$ Let $\alpha\in\irrp{T\, |\, \xi}{2}$ and write $\chi=\alpha{\uparrow^G}\in\irrp{G\, |\, \xi}{2}$. 
Then 
$$
[\chi, \varphi{\uparrow^G}]=[\chi, \gamma{\uparrow^G}]=[\chi, (\gamma{\uparrow^T})\big{\uparrow}^G]=[\alpha, \gamma{\uparrow^{T}}]\, ,
$$
where the last equality follows from Clifford correspondence because $\alpha, \gamma{\uparrow^T}\in\ch {T\, |\, \xi}$. 
Thus, $\chi$ lies over $\varphi$ if and only if $\alpha$ lies over $\gamma$. 
Similarly, for $\beta\in\irrp{H\cap T\, |\, \xi}{2}$, we have that 
$\beta{\uparrow^H}\in\irrp{H\, |\, \xi}{2}$ lies over $\nu$
if and only if $\beta$ lies over $\delta$. We deduce that 
$\, \hat{}\, $
restricts to a bijection
$$
\irrp{\, G|\, \varphi}{2}\longrightarrow \irrp{H |\, \nu}{2}\, ,
$$
as wanted. So we can and do assume from now on that $\xi$ is invariant in $G$. 
\end{proof}

\begin{Step}\label{fullyramified}
{We can assume that $\xi$ is fully-ramified in $K/L$.}
\end{Step}

\begin{proof}[Proof of Step  \ref{fullyramified}]
Let $L\subseteq J\subseteq K$ be the largest subgroup of $K$ such that $\xi$ extends to $J$ and every extension of $\xi$ to
$J$ is fully-ramified in $K$ (see Theorem 2.7 of \cite{brown}). Since $\xi$ is $G$-invariant, it is clear that $J\nor G$. By Problem $13.5$ of \cite{isaacs}, there exists a unique extension $\phi\in\irr{J}$ of $\xi$ which is $P$-invariant. Since $PJ\nor HJ$, it easily follows from the uniqueness of $\phi$ that this character is invariant in $HJ$. Then, by Lemma 10.5 of \cite{brown}, for any $P$-invariant subgroup $J\subseteq X\subseteq HJ$, restriction of characters defines a bijection

\begin{equation}\label{restriction}
\irr{X\, |\, \phi}\longrightarrow \irr{X\cap H\, |\, \xi}\, ,
\end{equation}
and it is immediate that this bijection maps $P$-invariant characters onto $P$-invariant characters. By uniqueness of $\phi$, an easy counting argument on character degrees yields $\irrp{X\, |\, \xi}{2}\subseteq \irr{X\, |\, \phi}$, for any $X$ as above. It follows from this and  Lemma \ref{glauberman} that restriction of characters defines a bijection
$$
\irrp{HJ\, |\, \xi}{2}\longrightarrow \irrp{H |\, \xi}{2}\, 
$$
satisfying the condition required for the correspondence $f_\xi$ in the statement. Thus, we may assume that $L=J$, which is the same to say that 
$\xi$ is fully-ramified in $K/L$. 

\end{proof}
\setcounter{Step}{0}

By the previous Step, $\xi{\uparrow^K}$ has a unique irreducible constituent, which we denote by $\theta$. We observe that $\theta$ is $P$-invariant, by 
uniqueness.

\medskip 
\begin{FStep}
\end{FStep} We are in a situation in which Theorem 9.1 of \cite{brown} applies. This Theorem gives $U\subseteq G$ such that $G=KU$ and $K\cap U=L$.
By Glauberman's Lemma we may assume that $U$ is $P$-invariant. Applying the Frattini argument (see the last paragraph of p. $632$ in \cite{brown}), we can assume that $U=H$. Let $\Psi$ be the character of $G/L$ given by Theorem 9.1. In particular the equation $\chi{\downarrow_U}=(\Psi{\downarrow_U}) \hat{\chi}$, for $\chi\in\irr{G\, |\, \theta}$ and $\hat{\chi}\in\irr{H\, |\, \xi}$ defines a choice-free one-to-one correspondence between these sets of characters. The uniqueness of $\theta$ gives $\irr{G\, |\, \xi}= \irr{G\, |\, \theta}$, and since $\Psi(1)=\sqrt{|K/L|}$ is odd we obtain a bijective correspondence
$$
\hat{}\, :\irrp{G\, |\, \xi}{2}\longrightarrow \irrp{H |\, \xi}{2}\, .
$$
In order to see that $\, \hat{}\, $ satisfies the desired property, 
let $K\subseteq V\subseteq G$ be a $P$-invariant odd-order subgroup of $G$.
Suppose that $W\subseteq G$ is such that $W/L$ complements $K/L$ in $V/L$
and satisfies the conditions of Theorem 9.1 of \cite{brown} for the character five $(V,K,L,\theta, \xi)$. 
Since $K/L$ acts transitively on the $G$-conjugacy class of $W$,
Glauberman's Lemma allows us
to assume that $W$ is $P$-invariant. Since $PK/K\nor G/K$
and $(|PK/K|, |V/K|)=1$, it follows
that $[W,P]\subseteq K\cap W= L$. Then 
$W/L \subseteq \cent{V/L}{P}\subseteq (H\cap V)/L$. 
Thus $W=H\cap V$. Arguing as before and using the inductive definition of 
the character $\Psi$ (see {p.~619} of \cite{brown}), we deduce that the equation
$\varphi{\downarrow_W}=(\Psi{\downarrow_W})\nu$, where $\varphi\in\irrp{V\, |\, \xi}{2}$
and $\nu\in\irrp{H\cap V\, |\, \xi}{2}$, defines a bijection between
these sets of characters. Furthermore, it is clear that 
a $P$-invariant $\varphi\in\irrp{V\, |\, \xi}{2}$ corresponds 
to a $P$-invariant $\nu\in\irrp{H\cap V}{2}$, and since $|V|$ is odd,
this occurs if and only if $[\varphi{\downarrow_{V\cap H}}, \nu]$ is odd, by Theorem 9.1 of \cite{brown}, 
which in turn is equivalent to
$\varphi^*=\nu^*$,
by Lemma \ref{glauberman}.
Suppose that $\chi\in\irrp{G\, |\, \varphi}{2}$, and write
$$
\chi{\downarrow_V}=e\varphi+\Delta\, ,
$$
where $\Delta$ is either zero, or $\Delta\in\ch{V\, |\, \xi}$ does not contain $\varphi$ as a constituent. Then

$$
(\Psi{\downarrow_H}\hat{\chi})\big{\downarrow_W}=\chi{\downarrow_W}=(e\varphi+\Delta){\downarrow}_W=
e(\Psi{\downarrow_W})\nu+\Delta{\downarrow_W}\, ,
$$
and Theorem 9.1 of \cite{brown} implies that $\hat\chi$
lies over $\nu$, since $\hat{\chi}{\downarrow_W}\in\ch {W\, |\, \xi}$. 
Similarly, if we assume that $\hat\chi$ lies over $\nu$, then
$\chi$ lies over $\varphi$, and it follows that the map $\, \hat{}\, $
restricts to a bijection
$$
\irrp{\, G|\, \varphi}{2}\longrightarrow \irrp{H |\, \nu}{2}\, .
$$ The proof is complete.
\end{proof}

We want our correspondences of characters to be invariant
under the action of suitable Galois automorphisms and group
automorphisms.
More precisely, assuming the notation in Lemma \ref{bottom}, 
if $a\in\Aut (G)$ stabilizes the subgroups $K$ and $L$ of $G$, 
then $\xi^a\in\irr L$
is $P^a$-invariant, and we obtain a bijection
$$
f_{\xi^a}\, :\irrp{G\, |\, \xi^a}{2}\longrightarrow \irrp{H^a |\, \xi^a}{2}\, . 
$$
In this situation, we would like to have that
$$
f_{\xi}(\chi)^a=f_{\xi^a}(\chi^a)\, ,
$$for any $\chi\in\irrp{G\, |\, \xi}{2}$. Similarly, if 
$\sigma\in\Gal(\Q( \omega)/\Q)$, where $\omega\in\mathbb C$
is a primitive \linebreak {$|G|$-th} root of unity, then $\xi^\sigma\in\irr L$ is $P$-invariant
and Lemma \ref{bottom} provides a choice-free bijection
$$
f_{\xi^\sigma}\, :\irrp{G\, |\, \xi^\sigma}{2}\longrightarrow \irrp{H |\, \xi^\sigma}{2}\, .
$$
As before, we would like the map $f_{\xi^\sigma}$ to satisfy
$$
f_{\xi}(\chi)^\sigma=f_{\xi^\sigma}(\chi^\sigma)\, ,
$$
for any $\chi\in\irrp{G\, |\, \xi}{2}$. 

\medskip

A careful analysis of the proof of Lemma \ref{bottom} shows that the bijection $f_\xi$
in that result is built upon the following three correspondences of characters:
Clifford's correspondence (as in  Step \ref{invariant} of the proof); the correspondence
in Lemma $10.\,5$ of \cite{brown}, given by character restriction (as in Step \ref{fullyramified} of the proof);
and finally, the correspondence of Theorem 9.1 in \cite{brown}. 
The two former correspondences are easily seen to be 
preserved by Galois action and action by automorphisms of $G$
as described above. Thus, 
in order to see that the correspondence in Lemma \ref{bottom}
also satisfies this, we need to prove that the correspondence
in Theorem 9.1 of \cite{brown} is invariant under the actions
described in the previous paragraph. This is the content of the following result. 

\begin{proposition}\label{natural}
Let $(G,K,L,\theta, \phi)$ be a character five, and assume that $|K/L|$ is odd. 
Let $U\subseteq G$ be as in the conclusion of Theorem 9.1
of \cite{brown}. 

\begin{enumerate}

\item Suppose that $\sigma\in\Gal(\Q(\omega)/\Q)$, where 
$\omega\in\mathbb C$ has order $|G|$. 
If $\chi\in\irr{G\, |\, \theta}$ corresponds to $\xi\in\irr{U\, |\, \phi}$ 
under the bijection of Theorem 9.1
of \cite{brown}, then $\chi^\sigma$ corresponds to $\xi^\sigma$
under this same bijection when 
defined with respect to the character five $(G,K,L,\theta^\sigma, \phi^\sigma)$.
 
\item Suppose that $a\in\Aut(G)$ stabilizes the subgroups $K, L$ of $G$. 
If $\chi\in\irr{G\, |\, \theta}$ corresponds to $\xi\in\irr{U\, |\, \phi}$ 
under the bijection of Theorem 9.1
of \cite{brown}, then $\chi^a$ corresponds to $\xi^a\in\irr{U^a}$ under this same bijection 
when defined with respect to the character five $(G,K,L,\theta^a, \phi^a)$.

\end{enumerate}

\end{proposition}

\begin{proof}
Let $\Psi\in\ch G$ be defined as in Theorem $9.1$ of \cite{brown},
 with respect to the form $\langle\langle \, ,\, \rangle\rangle_{\phi}$.
We first prove (2). 
It is clear that \cite[Theorem 9.1]{brown}
appplies to the character five $(G, K, L, \theta^a, \phi^a)$, 
and we may choose $U^a\subseteq G$ as a representative of the $G$-conjugacy class
of subgroups of $G$ in the conclusion of that theorem. 
Suppose that $\Psi_a\in\ch G$ is the character computed with respect to the 
form $\langle\langle \, ,\, \rangle\rangle_{\phi^a}$ as in \cite[Theorem 9.1]{brown}. 
Since
$$
(\chi^a){\downarrow_{U^a}}=(\chi{\downarrow_U})^a=(\Psi{\downarrow_U}\xi)^a=(\Psi{\downarrow_U})^a\xi^a=(\Psi^a){\downarrow_{U^a}}\xi^a\, ,
$$
it suffices to prove that $\Psi_a=\Psi^a$. By the algorithm given in \cite[p.\,619 and p.\,626]{brown}
to compute $\Psi$, this follows immediately from the fact  that for any $s\in U$
and any Sylow $p$-subgroup $(K/L)_p$ of $K/L$, we have 
$$
\sum_{y\in{(K/L)}_p} {\langle\langle y , y^{a{s}a^{-1}} \rangle\rangle_{\phi}}^2=
\sum_{z\in{(K/L)_p}} {\langle\langle z^{a^{-1}} , z^{sa^{-1}} \rangle\rangle_{\phi}}^2=
\sum_{z\in{(K/L)_p}} {\langle\langle z , z^s \rangle\rangle_{\phi^{a}}}^2\, ,
$$
where the last identity is a consequence of the definition of the form $\langle\langle \, ,\, \rangle\rangle_{\phi}$
on \cite[p.\,596]{brown}. In order to prove (1), we argue similarly and notice that
for any $s\in U$, we have
$$
\big{(}\sum_{y\in{(K/L)_p}} {\langle\langle y , y^{s} \rangle\rangle_{\phi}}^2\big{)}^\sigma=
\sum_{y\in{(K/L)_p}} \big{(}{\langle\langle y , y^{s} \rangle\rangle_{\phi} }^\sigma\big{)}^2=
\sum_{y\in{(K/L)_p}} {\langle\langle y , y^{s} \rangle\rangle_{\phi^\sigma} }^2\, .
$$
\end{proof}

We remark that the proof of Lemma \ref{bottom} has the same structure that the proof of Theorem 10.6(ii)
of \cite{brown}. It is thus clear that the correspondence of characters of that theorem is invariant under
Galois action and action by group automorphisms as above. Of course, this was already evident
from the arguments in \cite{brown}.

Our next result implies the maximal-defect case of Theorem \ref{AMK}. 
The general case follows from the reduction to this case given by Fong-Reynolds theory, as discussed prior to Lemma \ref{bottom}.

\begin{teo}\label{main}
Let $G$ be a finite solvable group, $P\in\syl{2}{G}$ and $M\nor G$ with $|M|$ odd.
Assume that $\theta\in\irr M$ is $G$-invariant, and let $\theta^*\in\irr{\cent M P}$ be its Glauberman
correspondent for the action of $P$ on $M$. Then there exists a choice-free bijection
$$
\irrp{G\, |\, \theta}{2}\longrightarrow \irrp{\norm G P \, |\, \theta^*}{2}\, .
$$
\end{teo}

\begin{proof} We divide the proof of the theorem in several steps. 
Write $N=\norm G P$.

\medskip
\begin{Step}\label{relative}
We can assume that $G=NM$. 
\end{Step}

\begin{proof}[Proof of Step \ref{relative}]
We shall see that there exists a choice-free bijective correspondence
$$
\irrp{G\, |\, \theta}{2}\longrightarrow \irrp{NM\, |\, \theta}{2}\, .
$$
Assume that $NM<G$. Let $M\subseteq S\nor G$ be such that $S/M=O^{2'2}(G/M)$. 
By the Frattini argument, we have that $G=NS$. Thus, by assumption $M<S$.
Let $M\subseteq J\nor G$ be such that $J/M=(S/M)'$, and note that
$J<S$, because $G$ is solvable. Note also that $|S/J|$ is odd, 
and thus $\cent{S/J}{P}=1$. Let $H=NJ$, and observe that $G=HS$ and $H\cap S=J$.
We shall define a choice-free correspondence from 
$\irrp{G\, |\, \theta}{2}$ into $\irrp{H\, |\, \theta}{2}$, 
and then repeated applications of the same argument will yield the result. 

Let $\chi\in\irrp{G\, |\, \theta}{2}\bigcup\irrp{H\, |\, \theta}{2}$. 
Of course, all the irreducible constituents of $\chi{\downarrow_J}$ have odd degree. 
Since $\theta$ is $G$-invariant, it is also clear that
every irreducible constituent of $\chi{\downarrow_J}$ lies over $\theta$.
Observe that $P$ acts on the irreducible constituents of $\chi{\downarrow_J}$ by conjugation, 
and thus since $\chi(1)$ is odd, an easy counting argument implies that there exists
a $P$-invariant irreducible constituent of $\chi{\downarrow_J}$. Now, by Theorem 10.6(ii) of \cite{brown},
for each $P$-invariant $\varphi\in\irrp{J\,|\, \theta}{2}$, there exists a natural correspondence
$$
F_{\varphi} : \irrp{G\, |\, \varphi}{2}\longrightarrow \irrp{H\, |\, \varphi}{2}\, .
$$
If $\chi\in\irrp{G\, |\, \varphi_1}{2}\bigcap\irrp{G\, |\, \varphi_2}{2}$, 
where $\varphi_i\in\irrp{J\, |\, \theta}{2}$ is $P$-invariant for $i=1,2$, we claim that
$F_{\varphi_1}(\chi)=F_{\varphi_2}(\chi)$. Indeed, if 
$\varphi_1=(\varphi_2)^x$ for some $x$ in $G$ and we let $I$ be the inertia subgroup of 
$\varphi_1$ in $G$, then since $\varphi_1$ is $P$-fixed, we have that $P$ is contained in $I$, and since $\varphi_2$ is 
$P$-fixed also $P^{x}$ is contained in $I$. Now there exists $y$ in $I$ such that 
$n=yx^{-1}$ normalizes $P$, and $(\varphi_1)^n=\varphi_2$. Since the maps $F_{\varphi_i}$
are invariant under automorphisms of $G$ induced by $N$-conjugation, we have
$$
F_{\varphi_2}(\chi)=F_{\varphi_2}(\chi)^{n^{-1}}=F_{\varphi_1}(\chi^{n^{-1}})=F_{\varphi_1}(\chi)\, ,
$$
as claimed. Thus, the union $F$ of the maps $F_\varphi$, where $\varphi\in\irrp{J\, |\, \theta}{2}$ and 
$\varphi$ is $P$-invariant, is a well-defined map. By the above observations, it is clear that 
$$
F:\irrp{G\, |\, \theta}{2}\longrightarrow \irrp{H\, |\, \theta}{2}\, ,
$$
and it is easy to check that $F$ is bijective. Finally, note that the map $F$ is choice-free, 
because the maps $F_{\varphi}$ are choice-free. This completes the proof of Step \ref{relative}. 
\end{proof}

Let $K=[M,P]$, and note that $M=\cent M P K$ by coprime action. In particular $G=N K$. Let $L=K'<K$
and note that $\cent {K/L}{P}=1$, again by coprime action. Write $H=NL$, and notice that $H\cap K=L$. 
Denote by $\nu\in\irr{M\cap H}$ the unique $P$-invariant character such that $\nu^*=\theta^*$. Let
$y\in N$. Since the Glauberman map is invariant under automorphisms of $G$ stabilizing  both $P$ and $M$,
and in this case also stabilizing $M\cap H$, we have that
$$
(\nu^y)^*=(\nu^*)^y=(\theta^*)^y=(\theta^y)^*=\nu^*\, .
$$
It follows that $\nu^y=\nu$ by uniqueness in Glauberman correspondence, 
and thus $\nu$ is $H$-invariant. Then by induction on $|G|$ we obtain a choice-free bijection
$$
\irrp{H\, |\, \nu}{2}\longrightarrow \irrp{N \, |\, \theta^*}{2}\, .
$$
It is now clear that in order to finish the proof it suffices to prove the following: 

\medskip
\begin{FStep}\label{final}
There exists a choice-free correspondence
$$
f\, :\irrp{G\, |\, \theta}{2}\longrightarrow \irrp{ H \, |\, \nu}{2}\, .
$$ 
\end{FStep}

\begin{proof}[Proof of Final Step]
Recall that by Lemma \ref{glauberman} we have
$[\theta{\downarrow_{M\cap H}}, \nu]\neq 0$. Since $\nu$ is $P$-invariant
of odd degree, let $\xi\in\irr L$ be a $P$-invariant constituent of
$\nu{\downarrow_L}$. By Lemma \ref{bottom}, there exists a choice-free bijection 
$f_\xi$, depending only on $\xi$:
$$
f_\xi \, :\irrp{G\, |\, \xi}{2}\longrightarrow \irrp{H |\, \xi}{2}\, ,
$$
such that it restricts to a bijection:
$$
f_\xi \, :\irrp{G\, |\, \theta}{2}\longrightarrow \irrp{H |\, \nu}{2}\, .
$$
If $\xi^y$ is any other $P$-invariant contituent of $\nu{\downarrow_L}$, where $y\in M\cap H$,
then Lemma \ref{bottom} provides a natural correspondence
$$
f_{\xi^y}\, : \irrp{G\, |\, \xi^y}{2}\longrightarrow \irrp{H |\, \xi^y}{2}\, ,
$$
which again restricts to a bijection
$$
f_{\xi^y} \, :\irrp{G\, |\, \theta}{2}\longrightarrow \irrp{H |\, \nu}{2}\, .
$$
By Proposition \ref{natural} and the comments before it, 
for any $\chi\in\irrp{G\, |\, \theta}{2}$ we have that
$$
f_\xi(\chi)=f_\xi(\chi)^y=f_{{\xi}^y}(\chi^y)=f_{{\xi}^y}(\chi)\, .
$$
Thus, the restriction of $f_\xi$ to $\irrp{G\, |\, \theta}{2}$
is independent of $\xi$, and we obtain
a choice-free correspondence of characters $f$ as desired.
\end{proof}
\end{proof}




By Proposition \ref{natural} and the construction of the bijections leading to Theorem \ref{AMK},
it is clear that the correspondence in that theorem is preserved by both Galois action and action
by group automorphisms of $G$ stabilizing the block $B$. More precisely, assuming the notation in
Theorem \ref{AMK}, if $\sigma\in\Gal(\Q( \omega)/\Q)$, where $\omega\in\mathbb C$ is a primitive
$|G|$-th root of unity, is such that $B^\sigma=B$, then $b^\sigma=b$. Now, if $\, \hat{} \, $ is
the bijection in Theorem \ref{AMK}, then $\, \hat{} \, $ commutes with the actions of
$\langle \sigma\rangle$ on $\irr{B}$ and $\irr{b}$:
$$
{\hat\chi}^\sigma=\widehat{{\chi}^\sigma}\, ,
$$
for any $\chi\in\irr{B}$ of height zero. In particular, this implies that 
the height-zero irreducible characters in $B$ have the same fields of values as the height-zero irreducible 
characters in $b$.
Similarly, if $a\in \Aut(G)$ is such that $B^a=B$, then $a$
induces a height-preserving permutation $\kappa_a$ on $\irr B$, and a  height-preserving bijection
$\tau_a\, :\irr b\longrightarrow \irr{b^a}$, where $b^a\in{\rm Bl}(\norm {G}{D^a} )$. So we see
that the bijection $\, \hat{} \, $ in Theorem \ref{AMK} makes the following diagram commutative: 

\medskip
\begin{center}
\begin{tikzpicture}
  \matrix (m) [matrix of math nodes,row sep=3em,column sep=2em,minimum width=2em] {
     {\rm Irr}_0(B) & {\rm Irr}_0(B) \\
     {\rm Irr}_0(b)  & {\rm Irr}_0(b^a) \\};
  \path[-stealth]
    (m-1-1) edge node [left] {$\, \hat{} \, $} (m-2-1)
            edge node [above] {$\kappa_a$} (m-1-2)
    (m-2-1.east|-m-2-2) edge node [below] {$\tau_a$} (m-2-2)
    (m-1-2) edge node [right] {$\, \hat{} \, $} (m-2-2)
            (m-2-1);
\end{tikzpicture}
\end{center}

\section{Symmetric groups}\label{secsym}

Let $n$ be a natural number and let $\mathfrak{S}_n$ be the symmetric group on $n$ letters. The Alperin-McKay conjecture has been verified for symmetric groups with a beautiful argument by Olsson in \cite{Olsson}.
The main goal of the first part of the present section is to prove Theorem \ref{Thm:main}. 
In particular we determine a natural bijection $\chi\mapsto\chi^*$ between $\mathrm{Irr}_0(B)$ and $\mathrm{Irr}_0(b)$, where $b$ is the Brauer correspondent of a given $2$-block $B$ of $\fS_n$ with defect group $D$. This bijection is based on the leg-lengths of hooks of partitions and it is shown to be compatible with the restriction functor in the sense that $\chi^*$ is a constituent of the restriction $\chi{\downarrow_{\norm {\mathfrak{S}_n} D }}$.

In the second part of the section (see subsection \ref{subsLi}) we investigate in more detail the restriction to $\norm {\mathfrak{S}_n} D$ of any irreducible character in $\mathrm{Irr}(B)$. In Theorem \ref{theo:restriction} we show that given any irreducible character of $\mathfrak{S}_n$ lying in $B$, there exists $\psi\in\mathrm{Irr}_0(b)$ such that $\psi$ is a constituent of $\chi{\downarrow_{\norm {\fS_{n}} D}}$. Moreover, we characterize those irreducible characters lying in $B$ whose restriction to $\norm {\fS_{n}} D$ has a unique height-zero constituent. 

\subsection{Notation and background }\label{sec:notbackSn}
We start by recalling some basic facts in the representation theory of symmetric groups. We refer the reader to \cite{James}, \cite{JK} or \cite{OlssonBook} for a more detailed account. 
A partition $\lambda=(\lambda_1\geq\lambda_2\geq\dots\geq\lambda_\ell>0)$ is a finite non-increasing sequence of positive integers. We refer to $\lambda_i$ as a part of $\lambda$. We call $\ell=\ell(\lambda)$ the length of $\lambda$ and say that $\lambda$ is a partition of $|\lambda|=\sum\lambda_i$, written $\lambda\vdash|\lambda|$. It is useful to regard the empty sequence $(\,)$ as the unique partition of $0$. The Young diagram of $\lambda$ is the set $[\lambda]=\{(i,j)\in{\mathbb N}\times{\mathbb N}\mid 1\leq i\leq\ell(\lambda),1\leq j\leq\lambda_i\}$. We orient ${\mathbb N}\times{\mathbb N}$ with the $x$-axis pointing right and the $y$-axis pointing down, in the anglo-american tradition.

The conjugate of $\lambda$ is the partition $\lambda'$ such that $[\lambda']$ is the reflection of $[\lambda]$ in the line $y=x$. So $\lambda'$ has parts 
$\lambda_i'=|\{j\mid\lambda_j\geq i\}|$ and in particular $\lambda'_1=\ell(\lambda)$. We say that a partition $\mu$ is contained in $\lambda$, written $\mu\subseteq\lambda$, if $\mu_i\leq\lambda_i$, for all $i\geq1$. When this occurs, we call the non-negative sequence $\lambda\backslash\mu=(\lambda_i-\mu_i)_{i=1}^\infty$ a skew-partition, and we call the diagram
$[\lambda\backslash\mu]=\{(i,j)\in{\mathbb N}\times{\mathbb N}\mid 1\leq i\leq\ell(\lambda),\mu_i<j\leq\lambda_i\}$ 
a skew Young diagram.

The rim of $[\lambda]$ is the collection ${\mathcal R}(\lambda)=\{(i,j)\in[\lambda]\mid (i+1,j+1)\not\in[\lambda]\}$ of nodes on its south-eastern boundary. Given $(r,c)\in[\lambda]$, the associated rim-hook is $h(r,c)=\{(i,j)\in {\mathcal R}(\lambda)\mid r\leq i,c\leq j\}$.
Then $h=h(r,c)$ contains $e:=\lambda_r-r+\lambda'_c-c+1$ nodes, in $a(h)=\lambda_r-c+1$ columns
and $\lambda'_c-r+1$ rows. We call $\ell(h)=\lambda'_c-r$ the leg-length of $h$. We refer to $h$ as an $e$-hook of $\lambda$. The integer $e$ is sometimes denoted as $|h|$. Removing $h$ from $[\lambda]$ gives the Young diagram of a partition denoted $\lambda-h$. In particular $|\lambda-h|=|\lambda|-e$ and $h$ is a skew Young diagram. 

Let $h$ be an $e$ rim-hook which has leg-length $\ell$. 
The associated hook partition of $e$ is $\hat h=(e-\ell,1^\ell)$. So $(e-\ell,1^\ell)$ coincides with its $(1,1)$ rim-hook. Also there are $e$ hook partitions ${\mathcal H}(e)=\{(e),(e-1,1),\dots,(1^e)\}$ of $e$, distinguished by their leg-lengths $0,1,\dots,e-1$.

Now fix a positive integer $e$. 
We call a partition which has no $e$-hooks an $e$-core. For example the $2$-cores are the triangular partitions $\kappa_s=(s,s-1,\dots,2,1)$ for $s\geq0$. The $e$-core of $\lambda$ is the unique $e$-core $\kappa$ which can be obtained from $\lambda$ by successively removing $e$-hooks. We call the integer $(|\lambda|-|\kappa|)/e$ the $e$-weight of $\lambda$. The set $B(\kappa,w)$ of partitions of $n$ with $e$-core $\kappa$ and weight $w$ is called an $e$-block of partitions.

Notice that the hook-lengths in a single row or column of $\lambda$ are distinct. 
Let $(r,c)\in[\lambda]$. Then $\{|h(i,j)|\ :\ (i,j)\in[\lambda],i\ne r,j\ne c\}$ is a submultiset of the hook-lengths of $\lambda-h(r,c)$. So for $m\geq1$, a partition of $e$-weight less than $2m$ can have at most one $me$ rim-hook.

Recall that the cycle type of a permutation 
 $\sigma\in\mathfrak{S}_n$ is the partition whose parts are the sizes of the orbits of $\sigma$ on $\{1,2,\dots,n\}$. Now the complex irreducible characters of $\mathfrak{S}_n$ are naturally labelled by the partitions of $n$. Given any partition $\lambda$ of $n$ we denote by $\chi^\lambda$ the corresponding irreducible character of $\mathfrak{S}_n$. The following classical result can be iterated to find the values of these characters:

\begin{lemma}[Murnaghan-Nakayama Rule]
Let $\lambda\vdash n$ and let 
$(\sigma,\tau)\in\mathfrak{S}_e\times\mathfrak{S}_{n-e}$ such that $\sigma$ is an $e$-cycle. Then
$$
\chi^\lambda(\sigma\tau)=\sum_{h}(-1)^{\ell(h)}\chi^{\lambda-h}(\tau),
$$
where $h$ runs over all $e$-hooks in $\lambda$.
\end{lemma}

Let $p$ be a prime 
integer, and let $B$ be a $p$-block of $\mathfrak{S}_n$ with associated defect group $D$ (uniquely defined up to conjugacy in $\fS_n$). 
According to the famous result of Brauer and Robinson 
(as conjectured by Nakayama, see \cite{BrauerNakayama, RobinsonNakayama}), $B=B(\kappa,w)$ for some $p$-core $\kappa$ and weight $w\geq0$. Moreover, we can choose a defect group of $B$ to be a Sylow $p$-subgroup $P_{pw}$ of $S_{pw}$. Note that $\kappa$ is the unique partition in $B(\kappa,0)$. Hence an irreducible character $\chi^\lambda$ of $\fS_n$ is in $\mathrm{Irr}(B(\kappa, w))$ if and only if the $p$-core of $\lambda$ is $\kappa$. In this case we
set $h(\lambda)$ to be the ($p$-)height of $\chi^\lambda$.

\subsection{The Alperin-McKay bijection for $\fS_n$}\label{Sec:bijection}

From now on $p=2$, $w$ is a non-negative integer, $\kappa$ is a $2$-core and $n=|\kappa|+2w$. Write $2w=2^{w_1}+\cdots +2^{w_t}$, where $w_1>w_2>\cdots >w_t>0$. Then the Young 
subgroup $\mathfrak{S}_{2^{w_1}}\times\dots\times\mathfrak{S}_{2^{w_t}}$ contains a Sylow $2$-subgroup $P_{2w}=P_{2^{w_1}}\times P_{2^{w_2}}\times\dots\times P_{2^{w_t}}$ of ${\mathfrak S}_{2w}$.

\begin{lemma}\label{L:MacdonaldOllson1}
Let $\lambda\vdash n$ have $2$-core $\kappa$ and $2$-weight $w$. Then 
$\lambda$ has height-zero if and only if there is a sequence 
$\lambda=\lambda^{(1)}\supset\lambda^{(2)}\supset\dots\supset\lambda^{(t)}\supset\lambda^{(t+1)}=\kappa$ of partitions such that $\lambda^{(i)}\backslash\lambda^{(i+1)}$ is a $2^{w_i}$ rim-hook, for $i=1,\dots,t$. Such a sequence is unique, if it exists. 

Equivalently $\lambda$ has height-zero if and only if $\lambda$ has a unique removable $2^{w_1}$ rim-hook $h$ and $\lambda-h$ has height-zero.
\end{lemma}
\begin{proof}
It is clear that the last statement follows from the first.

We know that $\kappa=(s,s-1,\dots,1)$, for some $s\geq0$. 
Then the diagonal hook-lengths of $\kappa$ form a partition $\mu=(2s-1,2s-5,2s-9,\dots)$ of $|\kappa|$. Let 
$c\in{\mathfrak S}_{|\kappa|}$ have cycle type $\mu$. Then the Murnaghan-Nakayama rule implies that $\chi^\kappa(c)=(-1)^{\lfloor s^2/4\rfloor}$. Let $d\in{\mathfrak S}_{2w}$ have cycle type $(2^{w_1},2^{w_2},\dots,2^{w_t})$. So $d$ is a $2$-element which commutes with $c$. Then $\chi^\lambda(cd)$ and $\chi^\lambda(c)$ have the same parity, by standard character theory.

Character theory implies that $[\mathfrak{S}_n:C_{\mathfrak{S}_n}(c)]\chi^\lambda(c)/\chi^\lambda(1)$ is an integer, and the parity of this {\em central character} is independent of $\lambda\in B(\kappa,w)$, according to block theory. Now the defect group $P_{2w}$ of $B(\kappa,w)$ is a Sylow $2$-subgroup of $C_{\mathfrak{S}_n}(c)={\mathfrak S}_{2w}\times\langle c\rangle$. It follows that $\chi^\lambda(c)/2^{h(\lambda)}$ is an integer whose parity is independent of $\lambda$.

Suppose first that the given sequence $\lambda^{(i)}$ of partitions exists. Set $h_i=\lambda^{(i)}\backslash\lambda^{(i+1)}$. Then $\chi^{\lambda}(cd)=\prod(-1)^{\ell(h_i)}\chi^\kappa(c)$, by the Murnaghan-Nakayama rule. As $\chi^\kappa(c)$ is odd, we deduce that $\chi^\lambda(cd)$ is odd. But then $\chi^\lambda(c)$ is odd. So $\lambda$ has height $0$, by the previous paragraph.

Conversely suppose that $\chi^\lambda$ has height zero. Then $\chi^\lambda(c)$ is odd, by the previous two paragraphs (consider $\nu=(2w+s, s-1,s-2,\ldots, 1)$). 
This forces $\chi^\lambda(cd)\ne0$. So we can successively strip hooks of lengths $2^{w_1},2^{w_2},\dots$ from $\lambda$, according to the Murnaghan-Nakayama rule. Equivalently, the given sequence $\lambda^{(i)}$ of partitions exists. Moreover this sequence is unique, as $\lambda^{(i)}$ has $2$-weight strictly less than $2^{w_i}$, for $i=1,\dots t$.
\end{proof}

Following the above lemma, let $\lambda$ be a height-zero partition which has $2$-weight $w$ and associated sequence $\lambda=\lambda^{(1)}\supset\lambda^{(2)}\supset\dots\supset\lambda^{(t)}\supset\lambda^{(t+1)}=\kappa$. We call the $t$-tuple ${\mathcal H}(\lambda)=(\hat h_1,\dots,\hat h_t)$ of  hook partitions associated to $(h_1,\dots,h_t)$ the hook-sequence of $\lambda$.

Conversely, let $\gamma_1,\gamma_2,\dots,\gamma_t$ be hook partitions of $2^{w_1},2^{w_2},\dots,2^{w_t}$, respectively. By \cite[Theorem 1.1]{Besse} there is a unique partition $\mu^{(t)}\supset\kappa$ such that $\mu^{(t)}\backslash\kappa$ is a $2^{w_t}$ rim-hook associated to $\gamma_t$. Given $i>1$, suppose that we have constructed a sequence $\mu^{(i+1)}\supset\dots\supset\mu^{(t)}\supset\mu^{(t+1)}=\kappa$ of partitions such that $\mu^{(j)}\backslash\mu^{(j+1)}$ is a $2^{w_j}$ rim-hook associated to $\gamma_j$, for $j=i+1,\dots,t$, and shown that this sequence is unique. Then $\mu^{(i+1)}$ has $2$-weight $\sum_{j=i+1}^t2^{w_j}<2^{w_i}$. So $\mu^{(i+1)}$ is a $2^{w_i}$-core. Again using \cite[Theorem 1.1]{Besse}, there is a unique partition $\mu^{(i)}\supset\mu^{(i+1)}$ such that $\mu^{(i)}\backslash\mu^{(i+1)}$ is a $2^{w_i}$ rim-hook associated to $\gamma_i$. This shows that there is a unique partition $\mu=\mu^{(1)}$ which has $2$-weight $w$, $2$-core $\kappa$ and hook sequence $(\gamma_1,\dots,\gamma_t)$. In particular $\mu$ has height zero. Counting the number of hook sequences for $w$ gives:

\begin{corollary}
Let $B$ be a $2$-block of $\mathfrak{S}_n$ which has $2$-weight $w$. If $2w=\sum_{i=1}^t 2^{w_i}$ with $w_1>\dots>w_t>0$ then $B$ has $2^{w_1+\cdots+ w_t}$ height-zero irreducible characters.
\end{corollary}

\begin{example}
The $8$ height zero partitions of $9$ in $B((2,1),3)$, their hook-sequences and the sequence of leg-lengths of these hooks:
$$
\begin{array}{l|cccc}
 \lambda & (8,1) & (4,3,2) & (4,2^2,1) & (4,1^5)\\
 \hline
 {\mathcal H}(\lambda) & ((4),(2)) & ((3,1),(2)) & ((2,1^2),(2)) & ((1^4),(2))\\
 \mbox{leg-lengths} & (0,0) & (1,0) & (2,0) & (3,0)\\
 \\
 \lambda  & (6,1^3) & (4,3,1^2) & (3^2,2,1) & (2,1^7)\\
 \hline
 {\mathcal H}(\lambda) & ((4),(1^2)) & ((3,1),(1^2)) & ((2,1^2),(1^2)) & ((1^4),(1^2))\\
 \mbox{leg-lengths} & (0,1) & (1,1) & (2,1) & (3,1)
\end{array}
$$
\end{example}

\medskip

Recall that $P_2$ is cyclic of order $2$ and $P_{2^k}\cong P_{2^{k-1}}\wr C_2$ for all $k>1$. It is not hard to show that 
$\norm {\mathfrak{S}_{2^k}} {P_{2^k}}=P_{2^k}$. 
Now the odd-degree characters of $\mathfrak{S}_{2^k}$ are precisely those labelled by the hook partitions $\mathcal{H}(2^k)$ of $2^k$. Moreover, by \cite[Theorem 1.1]{G} there is a bijection between these characters and the linear characters of $P_{2^k}$; if $h$ is a hook partition of $2^k$, the corresponding character of $P_{2^k}$ is the unique linear constituent $\phi^{h}$ of $\chi^{h}{\downarrow_{P_{2^k}}}$.

As mentioned above, $P_{2w}$ is a defect group of the $2$-block $B(\kappa,w)$ of $\mathfrak{S}_{n}$. 
Now $P_{2w}$ has normalizer $N:=P_{2w}\times\fS_{|\kappa|}$ in $\mathfrak{S}_{n}$. The Brauer correspondent of $B$ is the unique $2$-block $b=b_0\times B(\kappa, 0)$ of $N$ such that $b^{\mathfrak{S}_n}=B$ in the sense of Brauer. Here $b_0$ is the unique $2$-block of the $2$-group $P_{2w}$ and $B(\kappa, 0)$ is a defect zero $2$-block of $\mathfrak{S}_{|\kappa|}$. It is easy to check that an irreducible character $\psi\times\chi^\kappa$ in $b$ has height-zero if and only if $\psi(1)=1$. Equivalently we must have
$$
\psi=\phi^{h_1}\times\phi^{h_2}\times\cdots\times\phi^{h_t},
$$
where $h_j$ is a hook partition of $2^{w_j}$, for all $j\in\{1,\ldots, t\}$.

\medskip

Let $\chi=\chi^\lambda$ be a height-zero character in $B(\kappa, w)$, and suppose that $\lambda$ has hook sequence ${\mathcal H}(\lambda)=(h_1,\ldots, h_t)$. We denote by $\chi^*$ the height-zero character in $b$ defined by
$$
\chi^*=\phi^{h_1}\times\phi^{h_2}\times\cdots\times\phi^{h_t}
\times\chi^\kappa.
$$

We are now able to prove Theorem \ref{Thm:main}. This follows from the following result. 

\begin{teo}\label{t:comb}
The map $\chi\mapsto\chi^*$ is a bijection between the height-zero characters in $B(\kappa, w)$ 
and the height-zero characters in its Brauer correspondent $b$. Moreover $\chi^*$ is a constituent of the restriction of\/ $\chi$ to $N$.   
\end{teo}

\begin{proof}
Set $\mathcal{H}(w)=\mathcal{H}(2^{w_1})\times\cdots\times\mathcal{H}(2^{w_t})$. The first assertion about the bijectivity of the map follows from the discussion above. There we explicitly described two bijections. The first one between the sets
$\mathrm{Irr}_0(B(\kappa, w))$ and $\mathcal{H}(w)$, the second one between the sets $\mathcal{H}(w)$ and $\mathrm{Irr}_0(b)$. The composition of this two bijections gives the map defined by $\chi\mapsto\chi^*$. 

Let $\chi=\chi^\lambda$ where $\lambda$ is a partition of $n$. To prove the second statement we proceed by induction on $t$, the length of the $2$-adic expansion of $2w$. Suppose then that $t=1$ and $2w=2^{w_1}$. Let $h$ be the hook partition of $2w$ corresponding to the unique $2w$ rim-hook of $\lambda$. A direct application of the Littlewood-Richardson rule shows that $\chi^h\times\chi^\kappa$ is a constituent of $\chi^\lambda{\downarrow_{\fS_{2w}\times \fS_{n-2w}}}$ (see \cite[Lemma 4.1]{GKNT} for an explicit proof). It follows that $\chi^*=\phi^h\times \chi^\kappa$ is a constituent of $\chi^\lambda{\downarrow_{P_{2w}\times \fS_{n-2w}}}$.

Let now $t\geq 2$ and suppose that $\lambda$ has hook sequence ${\mathcal H}(\lambda)=(h_1,h_2,\ldots, h_t)$. Let $\mu$ be the unique height zero partition of $n-2^{w_1}$ with $2$-core $\kappa$ and hook sequence ${\mathcal H}(\mu)=(h_2,\ldots, h_t)$. In particular $\mu$ belongs to $B(\kappa,w-2^{w_1-1})$. Again by \cite[Lemma 4.1]{GKNT}, $\chi^{h_1}\times\chi^\mu$ is a constituent of $\chi^\lambda{\downarrow_{\fS_{2^{w_1}}\times\fS_{n-2^{w_1}}}}$. Clearly $(\chi^{\mu})^*
=\phi^{h_2}\times\phi^{h_3}\times\cdots\times\phi^{h_t}\times\chi^\kappa$, and by induction we have that
$$
\phi^{h_2}\times\phi^{h_3}\times\cdots\times\phi^{h_t}\times\chi^\kappa\ \ \text{is a constituent of}\ \ \ \chi^\mu{\downarrow_{P_{2^{w_2}}\times\cdots\times P_{2^{w_t}}\times\fS_{n-2w}}}.
$$
We conclude that $\chi^*$ is a constituent of $\chi^\lambda{\downarrow_{P_{2w}\times \fS_{n-2w}}}$. 
\end{proof}



We end this section by observing that the map described in Theorem \ref{t:comb} can be equivalently defined in algebraic terms only, without using combinatorics. 
This is done via repeated applications of \cite[Theorem B]{INOT}. 
Keeping the notation introduced above let $n=2w+|\gamma|$ where $2w=2^{w_1}+\cdots +2^{w_t}$, and let $\chi\in\mathrm{Irr}_{0}(B(\kappa,w))$. 
For $i\in\{1,\ldots,t\}$ we let $n_i:=n_{i-1}-2^{w_i}$, where we set $n_0=n$. 
We set $\chi_0:=\chi$, and we define a sequence of irreducible characters $(\chi_1,\ldots, \chi_t)$ as follows. 
For $i\in\{1,\ldots,t\}$ let $\chi_{i}\in\mathrm{Irr}(\fS_{n_i})$ be the unique irreducible constituent of $\chi_{i-1}\downarrow_{\fS_{n_i}}$ appearing with odd multiplicity. 
Notice that $\chi_1$ is well defined since $\chi_0$ has height zero and therefore the partition $\lambda$ of $n$ labelling $\chi_0$ has a unique removable rim $2^{w_1}$-hook $\gamma_1$, by Lemma \ref{L:MacdonaldOllson1}. Hence by \cite[Theorem B]{INOT} $\chi_0\downarrow_{\fS_{n_1}}$ has a unique irreducible constituent $\chi_1$ appearing with odd multiplicity. 
Again \cite[Theorem B]{INOT} guarantees that the partition of $n_1$ labelling $\chi_1$ is obtained by removing $\gamma_1$ from $\lambda$, hence $\chi_1$ is an irreducible character of height zero in $B(\kappa,w-2^{w_1-1})$, by Lemma \ref{L:MacdonaldOllson1}. 
Iterating this same argument $i$ times, we deduce that $\chi_i$ is well defined and uniquely determined. 

A second application of \cite[Theorem B]{INOT} implies that there exists a unique $\theta_i\in\mathrm{Irr}(\fS_{n_i}\times\fS_{2^{w_i}})$ such that $\theta_i$ lies above $\chi_i$ and such that $\frac{\theta_i(1)}{\chi_i(1)}$ is odd. Therefore $\theta_i=\chi_i\times \rho_i$ for some uniquely defined $\rho_i\in\mathrm{Irr}_{2'}(\fS_{2^{w_i}})$. 

We now denote by $\phi_i$ the unique irreducible odd-degree constituent of $\rho_i\downarrow_{P_{2^{w_i}}}$ (well defined by \cite[Theorem 1.1]{G}) and we set 
$$\chi^{**}=\phi_1\times\phi_2\times\cdots\times\phi_t\times\chi^\kappa.$$ 

\begin{teo}\label{t:nocomb}
The map $\chi\mapsto\chi^{**}$ coincides with the map constructed in Theorem \ref{t:comb}.
\end{teo}
\begin{proof}
Repeated applications of \cite[Theorem B]{INOT} and of Lemma \ref{L:MacdonaldOllson1} show that for all $i\in\{1,\ldots, t\}$ we have that $\chi_i=\chi^{\lambda_i}$ where $\lambda_i$ is the partition of $n_i$ obtained by removing from $\lambda_{i-1}$ the unique $2^{w_i}$-rim hook $\gamma_i$. Moreover $\rho_i=\chi^{h_i}$ where $h_i$ is the hook partition of $2^{w_i}$ associated to $\gamma_i$.
We conclude that $\phi_i=\phi^{h_i}$ and therefore that $\chi^{**}=\chi^*$.
\end{proof}

\subsection{Restriction to the normaliser}\label{subsLi}

We continue with the notation that $\kappa$ is a $2$-core of $n-2w$. However we now assume that $B=B(\kappa,w)$ is a non-principal $2$-block of ${\mathfrak S}_n$. So $\kappa=(s,s-1,\dots,1)$ with $s>1$. Also $P_{2w}$ is a defect group of $B$, $N=P_{2w}\times{\mathfrak S}_{n-2w}$ is the normalizer of $P_{2w}$ in ${\mathfrak S}_n$ and $b$ is the $2$-block of $N$ that is the Brauer correspondent of $B(\kappa,w)$. In Section \ref{Sec:bijection} we have seen that there exists a bijection $\chi\mapsto\chi^*$ between the sets $\mathrm{Irr}_0(B)$ and $\mathrm{Irr}_0(b)$. We have also shown that $\chi^*$ is a constituent of the restriction of $\chi$ to $N$. 

\begin{center}
\noindent\textbf{Question.} Is $\chi^*$ the only constituent of $\chi{\downarrow_N}$ lying in $\mathrm{Irr}_0(b)$?
\end{center}

In \cite[Theorem 1.2]{G} the first author addressed the above question in the case where $B$ is the principal block of $\fS_n$. In particular the following statement holds:

\begin{proposition}\label{theo1.2:G}
Let $\chi\in\mathrm{Irr}(\fS_n)$, and let $P_n$ be a Sylow $2$-subgroup of $\fS_n$. 
\begin{enumerate}
\item[(i)] The restriction of\/ $\chi$ to $P_n$ has a linear
constituent. 
\item[(ii)] The restriction of\/ $\chi$ to $P_n$ has a unique linear constituent if and only
if\/ $\chi(1)=1$ or $\chi(1)$ is odd and $n$ is a power of\/ $2$.
\end{enumerate}
In particular, if\/ $\chi(1)$ is even then its restriction to $P_n$ has at least two linear constituents. 
\end{proposition}

In this section we will prove the following statement. 

\begin{teo}\label{theo:restriction}
For each irreducible character $\chi$ in $B$.
\begin{itemize}
\item[(i)] The restriction of\/ $\chi$ to $N$ has an irreducible constituent in $\mathrm{Irr}_0(b)$. 
\item[(ii)] The restriction of\/ $\chi$ to $N$ has a unique irreducible constituent in $\mathrm{Irr}_0(b)$ if and only if\/ $\chi=\chi^\lambda$ where $\lambda$ is either the maximal or the minimal partition in $B(\kappa,w)$
(with respect to the dominance order). 
\end{itemize}
In particular, if $\chi$ has positive height then its restriction to $N$ has at least two height-zero constituents lying in $b$. 
\end{teo}

The proof will be given in a series of results, culminating in Proposition \ref{P:main}. First recall the following important definition and rule. 

\begin{definition}
Let ${\mathcal A}=a_1,\dots,a_k$ be a sequence of positive integers. The type of $\mathcal{A}$ is the sequence of non-negative integers $m_1,m_2,\dots$ where $m_i$ is the number of occurrences of $i$ in $a_1,\dots,a_k$. We say that $\mathcal{A}$ is a \textit{reverse lattice sequence} if the type of its prefix $a_1,\dots,a_j$ is a partition, for all $j\geq1$. Equivalently, for each $j=1,\dots,k$ and $i\geq2$
$$
|\{u\mid 1\leq u\leq j, a_u=i-1\}|\geq |\{v\mid 1\leq v\leq j, a_v=i\}|.
$$
\end{definition}

Let $\alpha\vdash n$ and $\beta\vdash m$ be partitions. The outer tensor product $\chi^\alpha\times\chi^\beta$ is an irreducible character of ${\mathfrak S}_n\times{\mathfrak S}_m$. Inducing this character to ${\mathfrak S}_{n+m}$ we may write
$$
(\chi^\alpha\times\chi^\beta){\uparrow^{{\mathfrak S}_{n+m}}}=\sum_{\gamma\vdash(n+m)}c_{\alpha,\beta}^\gamma\chi^\gamma.
$$
The \textit{Littlewood-Richardson rule} asserts that $c_{\alpha,\beta}^\gamma$ is zero if $\alpha\not\subseteq\gamma$ and otherwise equals the number of ways to replace the nodes of the diagram $[\gamma\backslash\alpha]$ by natural numbers such that 
\begin{enumerate}
\item The numbers are weakly increasing along rows.
\item The numbers are strictly increasing down the columns.
\item The sequence obtained by reading the numbers from right to left and top to bottom is a reverse lattice sequence of type $\beta$.
\end{enumerate}
We call any such configuration a Littlewood-Richardson filling of $[\gamma\backslash\alpha]$.

Recall that a partition $\alpha$ dominates a partition $\beta$, written $\beta\preceq\alpha$, if $\sum_{j=1}^i\beta_j\leq\sum_{j=1}^i\alpha_j$, for all $i\geq1$. We will use $\lambda^0$ and $\lambda^1$ to denote the most dominant and least dominant partitions in $B(\kappa,w)$, respectively. So $\lambda^0$ and $\lambda^1$ are obtained by wrapping a horizontal, respectively a vertical $2w$ rim-hook onto $\kappa$. In particular $\lambda^0$ and $\lambda^1$ have height zero. 

\begin{lemma}
The following hold:
\begin{itemize}
\item[(i)] $\chi^{\lambda^0}$ is the unique irreducible $B(\kappa,w)$-constituent of $(\chi^{(2w)}\times\chi^\kappa){\uparrow^{{\mathfrak S}_n}}$ and $\chi^{\lambda^1}$ is the unique irreducible $B(\kappa,w)$-constituent of $(\chi^{(1^{2w})}\times\chi^\kappa){\uparrow^{{\mathfrak S}_n}}$.
\item[(ii)] $(\chi^{\lambda}){\downarrow_N}$ has a unique height zero irreducible constituent in $b$, for $\lambda=\lambda^0,\lambda^1$.
\end{itemize}
\begin{proof}
To prove (i), let $\lambda\in B(\kappa,w)$ such that $c_{(2w),\kappa}^\lambda\ne0$. Then $\lambda$ has $2$-core $\kappa$, and no two nodes in $[\lambda\backslash\kappa]$ belong to the same column, using the Littlewood-Richardson rule. However $[\lambda\backslash\kappa]$ is a union of 2-hooks. As $\kappa$ is triangular, this forces $[\lambda\backslash\kappa]$ to be a single row. So $\lambda=\lambda^0$ and $c_{(2w),\kappa}^\lambda=1$. The proof for $(\chi^\kappa\times\chi^{(1^{2w})}){\uparrow^{{\mathfrak S}_n}}$ is similar, and we omit it.

To prove (ii) let's first assume that $\lambda=\lambda^0$. Then $[\lambda\backslash\kappa]$ is a row. Using the Littlewood-Richardson rule, $\chi^\alpha\times\chi^\kappa$ is an irreducible constituent of the restriction of $\chi^\lambda$ to ${\mathfrak S}_{2w}\times{\mathfrak S}_{n-2w}$ if and only if $\alpha=(2w)$. Moreover, $\chi^{(2w)}\times\chi^\kappa$ occurs with multiplicity $1$ in the restricted character. It follows that the only constituent of $\chi^\lambda{\downarrow_N}$ lying in $b$ is
$$
\phi^{(2^{w_1})}\times \phi^{(2^{w_2})}\times \cdots\times\phi^{(2^{w_t})}\times \chi^\kappa.
$$
Moreover, this character appears with multiplicity $1$. A completely similar argument (replacing $(2w)$ with $(1^{2w})$) covers the case $\lambda=\lambda^1$. 
\end{proof}
\end{lemma}

Now let rows$(\gamma\backslash\alpha)$ be the partition obtained by sorting the row lengths of $\gamma\backslash\alpha$ into weakly decreasing order, and cols$(\gamma\backslash\alpha)$ the partition obtained from the column lengths.

\begin{lemma}[\cite{McNa}]
Replacing the nodes in each column of $[\gamma\backslash\alpha]$ with $1,2,\dots$, from top to bottom, produces a Littlewood-Richardson filling of $[\gamma\backslash\alpha]$ of type cols$(\gamma\backslash\alpha)^t$.

Likewise, replacing the nodes in the rightmost boxes of each non-empty row of $[\gamma\backslash\alpha]$ with $1,2,\dots$, from top to bottom, and repeating to exhaustion, produces a Littlewood-Richardson filling of $[\gamma\backslash\alpha]$ of type rows$(\gamma\backslash\alpha)$.
\end{lemma}

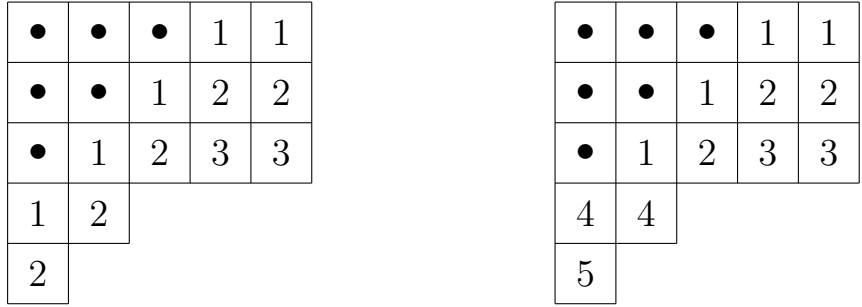
\begin{figure}[!htp]
\begin{center}

\begin{tikzpicture}[scale=0.4, every node/.style={transform shape}]
\tikzstyle{every node}=[font=\large]

\draw (2,10)--(12,10);
\draw (2,8)--(12,8);
\draw (2,6)--(12,6);
\draw (2,4)--(12,4);
\draw (2,2)--(6,2);
\draw (2,0)--(4,0);

\draw (2,10)--(2,0);
\draw (4,10)--(4,0);
\draw (6,10)--(6,2);
\draw (8,10)--(8,4);
\draw (10,10)--(10,4);
\draw (12,10)--(12,4);

\node at (3,9) {$\bullet$};
\node at (5,9) {$\bullet$};
\node at (7,9) {$\bullet$};
\node at (3,7) {$\bullet$};
\node at (5,7) {$\bullet$};
\node at (3,5) {$\bullet$};
\node at (9,9) {$1$};
\node at (11,9) {$1$};

\node at (7,7) {$1$};
\node at (9,7) {$2$};
\node at (11,7) {$2$};

\node at (5,5) {$1$};
\node at (7,5) {$2$};
\node at (9,5) {$3$};
\node at (11,5) {$3$};

\node at (3,3) {$1$};
\node at (5,3) {$2$};
\node at (3,1) {$2$};



\draw (20,10)--(30,10);
\draw (20,8)--(30,8);
\draw (20,6)--(30,6);
\draw (20,4)--(30,4);
\draw (20,2)--(24,2);
\draw (20,0)--(22,0);

\draw (20,10)--(20,0);
\draw (22,10)--(22,0);
\draw (24,10)--(24,2);
\draw (26,10)--(26,4);
\draw (28,10)--(28,4);
\draw (30,10)--(30,4);

\node at (21,9) {$\bullet$};
\node at (23,9) {$\bullet$};
\node at (25,9) {$\bullet$};
\node at (21,7) {$\bullet$};
\node at (23,7) {$\bullet$};
\node at (21,5) {$\bullet$};
\node at (27,9) {$1$};
\node at (29,9) {$1$};

\node at (25,7) {$1$};
\node at (27,7) {$2$};
\node at (29,7) {$2$};

\node at (23,5) {$1$};
\node at (25,5) {$2$};
\node at (27,5) {$3$};
\node at (29,5) {$3$};

\node at (21,3) {$4$};
\node at (23,3) {$4$};
\node at (21,1) {$5$};


\end{tikzpicture}
\end{center}
\caption{The left hand diagram is a Littlewood-Richardson filling of $[(5^3,2,1)\backslash(3,2,1)]$ of type ${\rm cols}((5^3,2,1)\backslash(3,2,1))^t=(5^2,2)$ and the right hand is a filling of type ${\rm rows}((5^3,2,1)\backslash(3,2,1))=(4,3,2^2,1)$.}
\label{fig:a}
\end{figure}
\vspace{.2cm}

\begin{lemma}\label{L:at_least_two}
Let $\lambda\in B(\kappa,w)$ with $\lambda\ne\lambda^0,\lambda^1$. Then the restriction of $\chi^\lambda$ to ${\mathfrak S}_{2w}\times{\mathfrak S}_{n-2w}$ has at least two irreducible constituents of the form $\chi^\alpha\times\chi^\kappa$. Moreover $\alpha\ne(2w),(1^{2w})$.
\begin{proof}
Given the hypothesis on $\lambda$, we have already shown that $\chi^{(2w)}\times\chi^\kappa$ and $\chi^{(1^{2w})}\times\chi^\kappa$ are not constituents of the restriction of $\chi^\lambda$ to ${\mathfrak S}_{2w}\times{\mathfrak S}_{n-2w}$.

Suppose for the sake of contradiction that the restriction of $\chi^\lambda$ to ${\mathfrak S}_{2w}\times{\mathfrak S}_{n-2w}$ has only one irreducible constituent of the form $\chi^\alpha\times\chi^\kappa$, where $\alpha$ is a partition of $2w$. \cite[Lemma 4.4]{BesKle} implies that $[\lambda\backslash\kappa]$ has shape $[\alpha]$, or has shape $[\alpha]$ rotated by $\pi$ radians.

Note that $[\lambda\backslash\kappa]$ is not left justified, as $\kappa$ is a triangular partition with at least two rows, and $\lambda$ is neither a row nor a column. So $[\lambda\backslash\kappa]$ does not have a partition shape.

Suppose then that $[\lambda\backslash\kappa]$ rotated by $\pi$-radians has a partition shape. Equivalently, there is a partition $\mu\subset\kappa$ such that $[\mu]$ is disjoint from $[\lambda\backslash\kappa]$ and $[\lambda\backslash\mu]$ is a rectangle. Then $[\kappa\backslash\mu|\geq3$, as $[\kappa]$ is triangular, and $[\lambda\backslash\kappa]$ has at least two rows and two columns. Every rectangular partition has 2-core $[\,]$ or $[1]$. So if we remove all $2$-hooks from $[\lambda\backslash\mu]$, we are left with a skew diagram with at most $1$ node. In particular we will have removed at least two nodes from $[\kappa]$. This contradicts our hypothesis that $\lambda$ has $2$-core $\kappa$. So this case is also impossible.
\end{proof}
\end{lemma}

\begin{proposition}\label{P:main}
For $\lambda\in B(\kappa,w)$, with $\lambda\ne\lambda^0,\lambda^1$, the number of height-zero constituents in the restriction of $\chi^\lambda$ to $N$ which belong to $b$ is:
\begin{itemize}
\item[(i)] two, if\/ $w$ is a power of $2$ and $[\lambda\backslash\kappa]$ is the disjoint union of a row and a column.
\item[(ii)] at least three, if\/ $w$ is a power of\/ $2$ and $[\lambda\backslash\kappa]$ is a $2w$ rim-hook of leg-length $1$ or $2w-2$.
\item[(iii)] at least four, in all other cases.
\end{itemize}
\end{proposition}
\begin{proof}
By Lemma \ref{L:at_least_two}, the restriction of $\chi^\lambda$ to ${\mathfrak S}_{2w}\times{\mathfrak S}_{n-2w}$ has at least two irreducible constituents of the form $\chi^\alpha\times\chi^\kappa$, with $\alpha\ne(2w),(1^{2w})$.

Assume the hypothesis of (i). Then the Littlewood-Richardson rule gives
$$
(\chi^\lambda){\downarrow_{{\mathfrak S}_{2w}\times{\mathfrak S}_{n-2w}}}=\big(\chi^{(m+1,1^{2w-m-1})}\times\chi^\kappa\big)+\big(\chi^{(m,1^{2w-m})}\times\chi^\kappa\big)+\psi
$$
where $m$ is the length of the row of $[\lambda\backslash\kappa]$ and no irreducible constituent of the character $\psi$ has the form $\chi^\alpha\times\chi^\kappa$. Proposition \ref{theo1.2:G}(ii) implies that the restriction of $\chi^\lambda$ to $N$ has exactly two height-zero irreducible constituents in $b$.

Suppose next that $w$ is a power of $2$ and ${\rm cols}(\lambda\backslash\kappa)^t$ is a hook partition. As $\lambda\ne\lambda^0,\lambda^1$, this means that $[\lambda\backslash\kappa]$ has a unique column of length $\geq2$. We may assume that $[\lambda\backslash\kappa]$ is not the disjoint union of a row and a column. As $\kappa$ is a non-trivial triangular partition, the only possibility remaining is that $[\lambda\backslash\kappa]$ is a $2w$ rim-hook of leg-length $1$. Then ${\rm rows}(\lambda\backslash\kappa)=(2w-2,2)$ is not a hook partition. Proposition \ref{theo1.2:G} implies that the restriction of $\chi^{(2w-2,2)}$ to $P_{2w}$ has at least two linear constituents. It follows that the restriction of $\chi^\lambda$ to $N$ has at least three height-zero irreducible constituents in $b$.

A similar argument works when $w$ is a power of $2$ and ${\rm rows}(\lambda\backslash\kappa)$ is a hook partition. In that case, we may assume that $[\lambda\backslash\kappa]$ is a $2w$ rim-hook of leg-length $2w-2$. Then ${\rm cols}(\lambda\backslash\kappa)^t=(2^2,1^{2w-4})$ is not a hook partition and once again Proposition \ref{theo1.2:G} implies that the restriction of $\chi^{(2^2,1^{2w-4})}$ to $P_{2w}$ has at least two height-zero irreducible constituents in $b$. This completes the analysis of the hypothesis of (ii).

To prove (iii), we may suppose that $w$ is not a power of $2$, or that neither ${\rm cols}(\lambda\backslash\kappa)^t$ nor ${\rm rows}(\lambda\backslash\kappa)$ are hook partitions. Proposition \ref{theo1.2:G} implies that the restriction of each $\chi^\alpha$ to $P_{2w}$ has at least two linear constituents. It follows that the restriction of $\chi^\alpha\times\chi^\kappa$ to $N$ has at least two height-zero irreducible constituents in $b$. Taking into account that there are at least two such $\alpha$, we see that the restriction of $\chi^\lambda$ to $N$ has at least four height-zero irreducible constituents in $b$.
\end{proof}

We are actually able to characterise when the number of height-zero constituents is $3$. As shown in Corollary \ref{Corlast} below, this situation occurs extremely rarely. 

We first need the following lemma. 

\begin{lemma}\label{Lem(-2,2)}
Let $k\geq 2$ be a positive integer. Then the linear character $\phi^{(2^k)}$ is an irreducible constituent of the restriction of $\chi^{(2^k-2,2)}$ to the Sylow $2$-subgroup $P_{2^k}$ of $\fS_{2^k}$. 
\end{lemma}
\begin{proof}
We proceed by induction on $k$. The statement is true for $k=2$ by direct computation. Suppose that $k>2$. For clarity we set $q=2^{k-1}$. Then $P_q\times P_q$ is a Sylow $2$-subgroup of the Young subgroup $\fS_q\times\fS_q$ of $\fS_{2^k}$, and we may assume that $P_q\times P_q\leq P_{2^k}$. The Littewood-Richardson shows that
\begin{equation}\label{E:rest}
\begin{aligned}
\chi^{(2^k-2,2)}{\downarrow_{\fS_q\times\fS_q}} =\hspace{1em}& \big(\chi^{(q)}\times\chi^{(q)}\big)+\big(\chi^{(q)}\times\chi^{(q-2,2)}\big)+\big(\chi^{(q-2,2)}\times\chi^{(q)}\big)\\
&+\big(\chi^{(q-1,1)}\times\chi^{(q-1,1)}\big)+\big(\chi^{(q)}\times\chi^{(q-1,1)}\big)+\big(\chi^{(q-1,1)}\times\chi^{(q)}\big)
\end{aligned}
\end{equation}
Taking into consideration Proposition \ref{theo1.2:G}, we get 
$$
\chi^{(2^k-2,2)}{\downarrow_{P_q\times P_q}}= (2c_{k-1}+1)\big(\phi^{(q)}\times\phi^{(q)}\big)+
\big(\phi^{(q-1,1)}\times \phi^{(q-1,1)}\big)+\Delta,
$$
where $c_{k-1}$ is the multiplicity of $\phi^{(q)}$ as an irreducible constituent of $\chi^{(q-2,2)}$ and $\Delta$ is a sum of irreducible characters of $P_q\times P_q$ all of the form $\eta\times\rho$ for some $\eta,\rho\in\mathrm{Irr}(P_q)$ with $\eta\neq\rho$. The inductive hypothesis guarantees that $c_{k-1}\neq 0$. 

Now \cite[Theorem 3.2]{G} shows that $\phi^{(2^{k})}$ and $\phi^{(2^{k}-1,1)}$ are the only linear characters of $P_{2^k}$ whose restriction to $P_q\times P_q$ equals $\phi^{(q)}\times\phi^{(q)}$. Likewise $\phi^{(2^{k}-2,1^2)}$ and $\phi^{(2^{k}-3,1^3)}$ are the only linear characters of $P_{2^k}$ whose restriction to $P_q\times P_q$ equals $\phi^{(q-1,1)}\times\phi^{(q-1,1)}$.
Suppose for the sake of contradiction that $\phi^{(2^{k})}$ is not a summand of $\chi^{(2^k-2,2)}\downarrow_{P_{2^k}}$. Then 
\begin{equation}\label{E:char}
\chi^{(2^k-2,2)}\downarrow_{P_{2^k}}= (2c_{k-1}+1)\big(\phi^{(2^{k}-1,1)})+
\phi^{(2^{k}-a,1^a)}+\Delta,
\end{equation}
where $a=2$ or $3$ and $\Delta$ is a sum of non-linear irreducible characters of $P_{2^k}$.

Now let $g\in P_{2^k}$ be a $2^k$-cycle. Then $\chi^{(2^k-2,2)}(g)=0$ by the Murnaghan-Nakayama rule, as $(2^k-2,2)$ has no rim-hooks of length $2^k$. It is shown in \cite[Theorem 3.2]{G} that $\phi^{(2^{k}-b,1^b)}(g)=(-1)^b$, for each $b\geq1$. Moreover it is easy to show that $\Delta(g)=0$, as $\Delta$ has no linear constituents. So \eqref{E:char} becomes
$$
0=-(2c_{k-1}+1)\pm 1.
$$
But $c_{k-1}$ is positive. So $-(2c_{k-1}+1)\pm 1<0$, leading to a contradiction.
\end{proof}

\begin{lemma}\label{4linear}
Let $k\geq 3$ be a positive integer. Then the restriction of $\chi^{(2^k-2,2)}$ to $P_{2^k}$ has at least four linear constituents.
\end{lemma}
\begin{proof}
We adopt the notation used in Lemma \ref{Lem(-2,2)}. Then \eqref{E:rest} implies that
$$
\chi^{(2^k-2,2)}\downarrow_{P_q\times P_q}= 3\cdot\big(\phi^{(q)}\times\phi^{(q)}\big)+
\big(\phi^{(q-1,1)}\times\phi^{(q-1,1)}\big)+\Omega,
$$
where $\Omega$ is a character of $P_q\times P_q$. The conclusion now follows from \cite[Lemma 2.2]{G}.
\end{proof}

\begin{corollary}\label{Corlast}
For each $\lambda\in B(\kappa,w)$ with $\lambda\ne\lambda^0,\lambda^1$ we have that the restriction of\/ $\chi^\lambda$ to $N$ has exactly three irreducible constituents in $\mathrm{Irr}_0(b)$ if and only if\/ $w=2$ and ${\rm rows}(\lambda\backslash\kappa)=(2,2)$ or ${\rm cols}(\lambda\backslash\kappa)^t=(2,2)$.
\end{corollary}
\begin{proof}
This follows from Lemma \ref{4linear} and the proof of Proposition \ref{P:main}.
\end{proof}

\bigskip

\noindent\textbf{Acknowledgments.}
We thank Pham Huu Tiep and Gabriel Navarro for several interesting conversations on the topic. 
Part of this work was done while the second and third author were visiting the University of Kaiserslautern. We are grateful to Gunter Malle for supporting these visits.

\end{document}